\documentclass{amsart}

\usepackage{amssymb,bbm}
\usepackage{amscd,amsfonts,amsmath,amssymb,bbm, dsfont}
\usepackage{enumerate,epsf,fancyhdr,float,graphicx,tabularx}
\usepackage{latexsym,mathrsfs,multirow}
\usepackage{wasysym}
\usepackage{xypic}
\usepackage[all]{xy}\usepackage[OT2,T1]{fontenc}
\usepackage[modulo,mathlines,displaymath, running]{lineno}
\usepackage{amsmath,mathrsfs,enumerate,wasysym}
\usepackage{xypic,hhline}
\usepackage[all]{xy}
\usepackage[OT2,T1]{fontenc}
\usepackage[modulo,mathlines,displaymath]{lineno}
\usepackage{tikz,tikz-cd}
\usepackage{dashrule}
\usetikzlibrary{matrix}
\usepackage[modulo,mathlines,displaymath]{lineno}

\newtheorem{theorem}{Theorem}[section]
\DeclareSymbolFont{cyrletters}{OT2}{wncyr}{m}{n}\DeclareMathSymbol{\Sha}{\mathalpha}{cyrletters}{"58}

\newcommand{\Fil}{\operatorname{Fil}}

\newcommand{\QQ}{\mathbb{Q}}

\newcommand{\cO}{\mathcal{O}}

\newcommand{\HIw}{H^1_{\mathrm{Iw}}}

\newcommand{\Dcris}{\mathbb{D}_{\rm cris}}
\newcommand{\Mlog}{M_{\log}}

\newcommand{\cyc}{\mathrm{cyc}}
\newcommand{\ur}{\mathrm{ur}}

\renewcommand{\Col}{\mathrm{Col}}
\newcommand{\uCol}{\underline{\mathrm{Col}}}
\renewcommand{\phi}{{\varphi}}
\newcommand{\Qp}{\mathbf{Q}_p}
\newcommand{\Zp}{\mathbf{Z}_p}

\newcommand{\ZZ}{\mathbf{Z}}

\renewcommand{\leq}{\leqslant}

\newcommand{\moon}{\bullet}

\newcommand{\links}{\left(\begin{array}{cc}}
	\newcommand{\rechts}{\end{array}\right)}
\newcommand{\bai}{\left[\begin{array}{cc}}
	\newcommand{\dai}{\end{array}\right]}
\newcommand{\hidari}{\left(\begin{array}{c}}
	\newcommand{\migi}{\end{array}\right)}

\newcommand{\Q}{\mathbb{Q}}

\newcommand{\Z}{\mathbb{Z}}

\newcommand{\ga}{{\mathfrak a}}

\newcommand{\p}{\mathfrak{p}}
\newcommand{\q}{\mathfrak{q}}

\newcommand{\calH}{\mathcal{H}}

\newcommand{\X}{{\mathcal X}}
\newcommand{\Y}{{\mathcal Y}}

\DeclareMathOperator{\rk}{rk}

\newcommand{\Gal}{\operatorname{Gal}}

\newcommand{\rank}{\operatorname{rank}}
\newcommand{\coker}{\operatorname{coker}}

\newcommand{\ord}{\operatorname{ord}}

\newcommand{\loc}{\operatorname{loc}}

\newcommand{\Sel}{\operatorname{Sel}}
\newcommand{\Tr}{\operatorname{Tr}}
\renewcommand{\O}{\mathcal{O}}

\newcommand{\directsum}{\oplus} % binary direct sum
\newcommand{\Union}{\bigcup} % union of a collection

\renewcommand{\contentsname}{Contents\\{\footnotesize\normalfont(A table
		of contents should normally not be included)}}

\newtheorem{auxiliary proposition}[theorem]{Auxiliary Proposition}

\newtheorem{corollary}[theorem]{Corollary}

\newtheorem{lemma}[theorem]{Lemma}
\newtheorem{main conjecture}[theorem]{Main Conjecture}
\newtheorem{main theorem}[theorem]{Main Theorem}
\newtheorem{modesty proposition}[theorem]{Modesty Proposition}

\newtheorem{open problem}[theorem]{Open Problem}

\newtheorem{proposition}[theorem]{Proposition}

\newtheorem{remark}[theorem]{Remark}

\newtheorem{convergence lemma}[theorem]{Convergence Lemma}
\newtheorem{corrected lemma}[theorem]{Corrected Lemma}
\newtheorem{growth lemma}[theorem]{Growth Lemma}
\newtheorem{coefficient lemma}[theorem]{Integrality Lemma}
\newtheorem{interpolation lemma}[theorem]{Interpolation Lemma}
\newtheorem{kernel lemma}[theorem]{Kernel Lemma}
\newtheorem{limit lemma}[theorem]{Limit Lemma}
\newtheorem{tandem lemma}[theorem]{Modesty Lemma}
\newtheorem{zero-finding lemma}[theorem]{Zero-Finding Lemma}

\renewcommand{\Y}{\mathcal{Y}}
\newcommand{\LL}{\mathcal{L}}
\newcommand{\vp}{\varphi}
\renewcommand{\cyc}{\mathrm{cyc}}
\renewcommand{\ur}{\mathrm{ur}}

\renewcommand{\P}{\mathfrak{P}}
\definecolor{Green}{rgb}{0.0, 0.5, 0.0}

\begin{document}

\title{Ranks of elliptic curves over $\Zp^2$-extensions}

\begin{abstract}
Let $E$ be an elliptic curve with good reduction at a fixed  odd prime $p$ and $K$ an imaginary quadratic field where $p$ splits. We give a growth estimate for the Mordell-Weil rank of $E$  over finite extensions inside the $\Z_p^2$-extension of $K$.
\end{abstract}

\author{Antonio Lei}
\address{Antonio Lei\newline
D\'epartement de math\'ematiques et de statistique\\
Universit\'e Laval, Pavillon Alexandre-Vachon\\
1045 Avenue de la M\'edecine\\
Qu\'ebec, QC\\
Canada G1V 0A6}
\email{antonio.lei@mat.ulaval.ca}

\author{Florian Sprung}
\address{Florian Sprung\newline
 School of Mathematical and Statistical Sciences\\
Arizona State University\\
Tempe, AZ 85287-1804\\ USA}
\email{florian.sprung@asu.edu}

\maketitle

\section{Introduction}
One fundamental question in the arithmetic of elliptic curves is to determine the Mordell-Weil rank of its rational points in a number field. While this question is hard, one can hope to answer a similar question in $p$-adic families of number fields: {\it How does the Mordell-Weil rank grow in the finite layers of the maximal $p$-unramified extension of a given number field $F$?}

When $F$ is abelian over $\Q$, works of Kato \cite{Ka} and Rohrlich \cite{rohrlich} show that the rank is bounded in the the cyclotomic $\Z_p$-extension, i.e. the $\Z_p$-extension contained in $\Union_nF(\zeta_{p^n})$. Thus, an answer to the question is "{\it It stays bounded}" when $F=\Q$.

The purpose of this article is to shed some light on the next simplest case, when $F$ is an imaginary quadratic field $K$. 
Such a $K$ can have $\Z_p$-extensions in which the Mordell-Weil rank may be unbounded. In fact, Mazur formulated a precise conjecture in the case of ordinary primes $p$, the \cite[Growth Number Conjecture]{mazuricm}. The Mordell-Weil rank at the $n$th layer of a $\Z_p$-extension of $K$ is known to be of the form $ap^n+\O(1)$, i.e. is governed mainly by an integer $a$, the growth number. The conjecture gives a recipe to choose $a\in\{0,1,2\}$ and predicts that $a=0$ unless the $\Z_p$-extension one is scrutinizing is the \textit{anticyclotomic} one. (The theorem of Kato--Rohrlich above is an instance for $a=0$.) In work towards this conjecture, the theorems of Cornut and Vatsal \cite{Cornut,Vatsal} combined with Bertolini's thesis \cite{bertolini} give scenarios in which $a=1$ along anticyclotomic $\Z_p$-extensions, so that in particular the Mordell-Weil rank is not bounded. See also Bertolini \cite{Be} for a more precise growth pattern of the rank. For a counterpart to Bertolini's result in the $a_p=0$ case, see \cite{longovigni}. In the case where $E$ has complex multiplication by $K$, we remark that the growth patterns for the ordinary and the supersingular case, as stated in \cite[discussion after Theorem 1.8]{Gr}, are different.

In this article, we develop methods to bound the Mordell-Weil rank in {\it all} subfields contained in the maximal $p$-unframified extension of $K$. The main innovation is that we present methods for both the ordinary and the supersingular case. 

\begin{theorem}
	Let $E$ be an elliptic curve over $\Q$, $K$ an imaginary field along which an odd prime $p$ coprime to the conductor of $E$ splits. Let $K_n$ be the subfield of the  $\Z_p^2$-extension of $K$ so that $\Gal(K_n/K)\cong(\Z/p^n\Z)^2$. Then $\rank E(K_n)=O(p^n)$.
\end{theorem}

We may thus speak of a ``$2$-dimensional growth number'' that governs the rank of $E(K_n)$ as one climbs along the two-dimensional tower of the $\Z_p^2$-extension of $K$. 
One strategy for proving a theorem like this is to employ a \textit{Control Theorem}. The idea of controlling, going back to Mazur, is the following: The $p^\infty$-Selmer group $\Sel_p(E/K_n)$ at finite level 
should be controlled by easier Selmer groups at infinite level $K_\infty$. We succeed in following this strategy in the cases $p \nmid a_p$ (i.e. $p$ is ordinary) or $a_p=0$ (which is a subcase of the supersingular case $p \mid a_p$) by applying control theorems by 
Greenberg ($p \nmid a_p$) and Kim ($a_p=0$), which generalize theorems of Mazur and Kobayashi concerning the cyclotomic $\Z_p$-extension, and combine them with a general growth lemma of Harris.

However, when $p\mid a_p\neq0$, there is no general control theorem available, so we combine partial controlling with a different strategy. Recall that the idea behind the control theorem was to control the Selmer group at finite level via Selmer groups at the infinite level. Present technology does not allow us to do this, but we are able to control a quotient $\X_n^0$, the \textit{fine} Selmer group (more precisely its dual), of the Selmer group dual $\X_n$. We deduce from this control theorem that the $\Z_p$-rank of $\X_n^0$ is $O(p^n)$. We thus have reduced the problem to showing that the $\Z_p$-rank of the kernel $\Y_n$ of the quotient map $\X_n\rightarrow \X_n^0$ is also $O(p^n)$.

The idea of shedding light on the Selmer group dual $\X_n$ via this kernel $\Y_n$ goes back to at least Kobayashi's methods of estimating the growth of $\Sha$ in cyclotomic towers. One of Kobayashi's key insights was to analyze $\Y_n$ via a similar but slightly more global Selmer-type group $\Y_n'$: While the defining condition for $\Y_n$ comes from the local points of $E$ inside the completions of $K_n$, $\Y_n'$ incorporates global information from the infinite level as well. (This extra global information is the $\Gal(K_n/K)$-coinvariants of the Iwasawa cohomology at $K_\infty$). From the point of view of $\Z_p$-ranks, $\Y_n'$  controls $\Y_n$ so that $\rk_{\Z_p}\Y_n\leq \rk_{\Z_p}\Y_n'$: Thus, we can control $\rk_{\Z_p}\X_n$, and thus the Mordell-Weil rank at level $n$ \textit{if we could handle $\rk_{\Z_p}\Y_n'$ directly}. What is new in this work is that we handle $\rk_{\Z_p}\Y_n'$ by relating it to Coleman maps that allows us to circumvent control theorems, and despite their absence show that $\rk_{\Z_p}\Y_n'=O(p^n)$.

To do this, we analyze the differences $\rk_{\Z_p}\Y_{n+1}'-\rk_{\Z_p}\Y_n'$ by relating them to the $\Z_p$-ranks of the local modules that contain the defining conditions of $\Y_n'$. We can break these local modules into (isotypical) components by a lemma of Cuoco and Monsky,  and find that the rank of each component is $c(p^{n+1}-p^n)$, where the possibilities for $c$ are $c=0,1,$ or $2$. The crucial step in this paper is to show that among these isotypical components in question, $c=1$ or $c=2$ happens only finitely many times, in fact a number of times that is \textit{bounded independently of $n$}. Thus, the jump in ranks sums up to at most $C(p^{n+1}-p^n)$ for a constant \footnote{an upper bound for the ``growth number''} $C$ independent of $n$, giving the estimate $\rk_{\Z_p}\Y_n'\leq Cp^n$. 

We accomplish this crucial step by showing that the localized Iwasawa cohomology \textit{vanishes} -- i.e. lies inside the Selmer condition -- if and only if it connects to a zero of a $p$-adic power series via pairs of Coleman maps $\Col_\sharp$ and $\Col_\flat$ and a certain two by two matrix (the logarithm matrix). The localized Iwasawa cohomology vanishes when $c\neq0$, while there can only be finitely many zeros of the power series it connects to. Therefore, we can have $c\neq0$ only a finite number of times.

\textbf{Outlook.} The `growth number conjecture'  \cite[Section 18]{mazuricm} is one instance of the philosophy of Greenberg, Mazur, and Rubin that growth in the Mordell-Weil rank should come from zeros forced by functional equations (cf. \cite{MR}), but concerns the case of ordinary reduction. We  write down a version that includes supersingular primes as well:

\newtheorem{openproblem}[theorem]{Growth Number Vanishing  Conjecture}
\begin{openproblem} Let $E$ be an elliptic curve with good reduction at a prime $p$. The Mordell-Weil rank stays bounded along any $\Z_p$-extension of $K$, unless the extension is anticyclotomic and the root number negative.  \end{openproblem}

\subsection*{Acknowledgment} The authors would like to thank Ralph Greenberg, Meng Fai Lim, Bharath Palvannan, and Karl Rubin for interesting discussions and helpful email correspondence during the preparation of this article.

\section{Setup and notation}
Throughout this article, $p$ is a fixed odd prime, $E$ is an elliptic curve defined over $\Q$ with good reduction at $p$ and the Tate module of $E$ at $p$ will be denoted by $T$. We also fix  an imaginary quadratic field $K$ over which $p$ splits, i.e.  $(p)=\p\q$ with $\p\neq\q$. We fix embeddings $\overline{\mathbf{Q}}\hookrightarrow \mathbf{C}$ and $\overline{\mathbf{Q}}\hookrightarrow \mathbf{C}_p$ so that $\p$ lands inside the maximal ideal of $\cO_{\mathbf{C}_p}$.

 Given an ideal $\ga$ in the ring of integers of $K$, we denote by $K(\ga)$ its ray class field.
Let $K_\infty$ be the $\Zp^2$-extension of $\Zp$ and write $\Gamma=\Gal(K_\infty/K)$.  If $n\ge0$ is an integer, we write $\Gamma_n=\Gamma^{p^n}$ and $$K_n=K_{\infty}^{\Gamma_n}=K(p^{n+1})^{\Gal(K(1)/K)}.$$  Note that $\Gamma\cong G_\p\times G_\q$, where $G_\star$ denotes the Galois group of the extension $\Gal(K(\star^\infty)\cap K_\infty/K)$. We fix topological generators $\gamma_\p$ and $\gamma_\q$ for these groups. The Iwasawa algebra $\Lambda:=\Zp[[\Gamma]]$ may be identified with the power series ring $\Zp[[\gamma_\p-1,\gamma_\q-1]]$. 
Given a $\Lambda$-module $M$, we shall write $M^\vee$ for its Pontryagin dual. We will employ the following theorem a number of times throughout the article.

\begin{theorem}\label{thm:harris}
Let $M$ be a finitely generated $\Lambda$-module of rank $r$. Then, 
\[
\rank_{\Zp}M_{\Gamma_n}=rp^{2n}+O(p^n).
\]
\end{theorem}
\begin{proof}
This is a special case of \cite[Theorem~1.10]{harris}.
\end{proof}

We assume throughout that there are only finitely many primes of $K_\infty$ lying over $p$. Let $p^t$ be the number of primes above $\p$ and $\q$.  We fix a choice of coset  representatives $\gamma_1,\ldots,\gamma_{p^t}$ and $\delta_1,\ldots,\delta_{p^t}$ for $\Gamma/\Gamma_\p$ and $\Gamma/\Gamma_\q$ respectively, where $\Gamma_\star$ denotes the decomposition group of $\star$ in $\Gamma$.

For an integer $n\ge 1$, we define
\[
\omega_n(X):=(1+X)^{p^n}-1 \text{, and the cyclotomic polynomial }\Phi_n(X):=\frac{\omega_n(X)}{\omega_{n-1}(X)}.
\]
We also define  $$\omega_n^+(X):=X\prod_{1\leq i \leq n, i \text{ even}} \Phi_{i}(X), \quad
\omega_n^-(X):=X\prod_{1\leq i  \leq n, i \text { odd}} \Phi_{i}(X).$$

\section{The ordinary case}
In this section, we assume that $E$ has good ordinary reduction at $p$ and we study the Mordell-Weil ranks of $E$ over $K_{n}$ as $n$ varies. The result is probably well known amongst experts. Nonetheless, we present the proof here to illustrate the divergence of the techniques employed in the supersingular case from the ordinary case.

Given an algebraic extension $F$ of $\QQ$,   the $p$-Selmer group of $E$ over $F$ is defined to be
\[
\Sel_p(E/F)=\ker\left(H^1(F,E[p^\infty])\rightarrow\prod_v\frac{H^1(F_v,E[p^\infty])}{E(F_v)\otimes \Qp/\Zp}\right),
\]
where $v$ runs through all places of $F$ in the product. The Selmer group over $K_\infty$ satisfies the following control theorem.
\begin{theorem}\label{thm:Greenbergcontrol}
Let $E/\QQ$ be an elliptic curve with good ordinary reduction at $p$. For an integer $n\ge$, let  $s_n$ denote the restriction map
\[
\Sel_p(E/K_{n})\rightarrow \Sel_p(E/K_{\infty})^{\Gamma_n}.
\]
Then, both $\ker(s_n)$ and $\coker(s_n)$ are finite.
\end{theorem}
\begin{proof}
This is a special case of \cite[Theorem 2]{greenberg}.
\end{proof}

\begin{theorem}\label{thm:ordranks}
Let $E/\QQ$ be an elliptic curve with good ordinary reduction at $p$. Then,
\[
\rank_{\Zp}\Sel_p(E/K_{n})^\vee=O(p^n)
\]
for $n\gg0$.
\end{theorem}
\begin{proof}
Let us first show that the dual Selmer group $\Sel_p(E/K_\infty)^\vee$ is torsion over $\Lambda$. Let $K_\cyc/K$ be the cyclotomic $\Zp$-extension of $K$ and write $\Gamma_\cyc$ and $H_\cyc$ for the Galois groups $\Gal(K_\cyc/K)$ and $\Gal(K_\infty /K_\cyc)$ respectively. Note that to show $\Sel_p(E/K_\infty)^\vee$ is $\Lambda$-torsion, it suffices to show that $\left(\Sel_p(E/K_\infty)^{H_\cyc}\right)^\vee$ is. As shown in \cite[proof of Proposition~8.4]{lp}, there is an isomorphism of $\Zp[[\Gamma_\cyc]]$-modules
\[
\Sel_p(E/K_\infty)^{H_\cyc}\cong \Sel_p(E/\QQ_\cyc)\oplus \Sel_p(E^{(K)}/\QQ_\cyc),
\]
where $\QQ_\cyc$ denotes the cyclotomic $\Zp$-extension of $\QQ$ and $E^{(K)}$ denotes the quadratic twist of $E$ by $K$. Both $\Sel_p(E/\QQ_\cyc)^\vee$ and $\Sel_p(E^{(K)}/\QQ_\cyc)^\vee$ are $\Zp[[
\Gamma_\cyc]]$-torsion by \cite[Theorem~17.4]{Ka} since both $E$ and $E^{(K)}$ have good ordinary reduction at $p$. Consequently, $\Sel_p(E/K_\infty)^\vee$ is $\Lambda$-torsion as claimed.

We may now combine Theorems~\ref{thm:harris} and \ref{thm:Greenbergcontrol} to deduce that
\[
\rank_{\Zp}\Sel_p(E/K_n)^\vee=\rank_{\Zp}\left(\Sel_p(E/K_\infty)\right)^\vee_{\Gamma_n}=O(p^n),
\]
hence the result.
\end{proof}

%\remark One can also prove this theorem by applying the control theorem \cite[Theorem 12]{mricm} and \cite[Theorem 17.4(i)]{Ka}.

\section{The $a_p=0$ case}\label{S:kim}

Let $\hat E$ denote the formal group attached to  $E$ over $\Z_p$.  Let  $\P$ be a fixed prime of $K(p^\infty)$ lying above $\p$. By an abuse of notation, we denote the prime $L\cap \P$ again by $\P$ whenever $L$ is a sub-extension of $K(p^\infty)/K$.  For non-negative integers $m$ and $n$, Generalizing Kobayashi's work in \cite{kobayashi03}, Kim defined in \cite{kimspaper} the following groups
\begin{align*}
  E^+(K(\p^m\q^n)_\P) & =\left\{P\in \hat E(K(\p^m\q^n)_\P):\Tr_{m/l+1,n}P\in \hat E(K(\p^l\q^n)_\P), 2|l<m\right\};  \\
  E^-(K(\p^m\q^n)_\P) & =\left\{P\in \hat E(K(\p^m\q^n)_\P):\Tr_{m/l+1,n}P\in \hat E(K(\p^l\q^n)_\P), 2\nmid l<m\right\},
\end{align*}
where $\Tr_{m/l+1,n}:\hat E(K(\p^m\q^n)_\P) \rightarrow \hat E(K(\p^{l+1}\q^n)_\P)$ denotes the trace map.   Furthermore, we define
\[
E^\pm((K_n)_\P):= E^\pm(K(p^{n+1})_\P)^\Delta\subset E((K_n)_\P),
\]
where $\Delta:=\Gal(K(p)/K(1))$. Given a prime $\mathfrak{Q}$ of $K(p^\infty)$ lying above $\q$, we may define $E^\pm((K_\infty)_\mathfrak{Q})$ similarly.
Let  $\bullet,\circ\in\{+,-\}$, then we have the signed Selmer group over $K_n$ given by:
\[
  \Sel_p^{\circ\moon}(E/K_n)=\ker\left(\Sel_p(E/K_n)\rightarrow\prod_{\P|\p}\frac{H^1((K_n)_\P,E[p^\infty])}{E^\circ((K_n)_\P)\otimes\Qp/\Zp}\times \prod_{\mathfrak{Q}|\q}\frac{H^1((K_n)_\mathfrak{Q},E[p^\infty])}{E^\moon((K_n)_\mathfrak{Q})\otimes\Qp/\Zp} \right).
\]
We have a similar definition for $\Sel_p^{\circ\moon}(E/K_n)$ on taking $E^\pm((K_\infty)_\star)$ to be the union of  $E^\pm((K_n)_\star)$.

\begin{theorem}[Kim's Control Theorem]\label{kimscontroltheorem}
Let $\circ,\moon\in\{+,-\}$. Then the natural module morphisms
$$\Sel^{\circ\moon}_p(E/K_n) \rightarrow \Sel^{\circ\moon}_p(E/K_\infty)[\omega_n^\circ(\gamma_\p)][\omega_n^\moon(\gamma_\q)]$$ have bounded kernel and cokernel as $n$ varies.
\end{theorem}

\begin{proof}This is \cite[Proposition 2.19]{kimspaper} with $\mathfrak{f}=1$ and $m=n$. Note that there is an assumption  $4\nmid[K(1)_v:\Qp]$ in \textit{loc. cit.}, where $v$ is a prime of $K(1)$ above $p$ when either $\circ$ or $\moon$ equals $+$. This assumption was used in the proof of Theorem~2.8 of \textit{op. cit.}, where certain freeness condition on $E^+$ is proved. However, since we are only interested in the base fields $K_\p$ and $K_\q$, the corresponding freeness condition is true and hence the control theorem we state is true following the same proof of \cite[Proposition 2.19]{kimspaper}. See also \cite[Theorem~3.34]{KO} where  a more general version of \cite[Theorem~2.8]{kimspaper} is proved.
\end{proof}

\begin{lemma}\label{lem:bigdiag}
Let $n\ge0$ and $\P|p$ a prime of $K_n$ above $p$. We write
\[
E^{\circ}_\star=\bigoplus_{\P|\p}E^\circ((K_n)_\P),\quad
E^1_\star=\bigoplus_{\P|\p}\hat E(K_\star),\quad
E^\forall_\star=\bigoplus_{\P|\p}\hat E((K_n)_\P),
\]
where $\star\in\{\p,\q\}$ is the prime of $K$ lying below $\P$.
Furthermore, we write
$$
E^{\circ\moon}=E_\p^\circ\oplus  E_\q^\moon
$$
for $\circ,\moon\in\{+,-,1,\forall\}$. Then, all horizontal and vertical sequences of the following commutative diagram are exact:
\[
\begin{tikzcd}
	&0\arrow[d] & 0\arrow[d, shift left=6]  \;\; \directsum  \;\; 0\arrow[d, shift right=6] &  0 \arrow[d]&\\
	0 \arrow[r] &E^{11} \arrow[d]\arrow [r]& E^{1+}\arrow[d, shift left=6]  \directsum  E^{1-}\arrow[d, shift right=6] \arrow[r] & E^{1\forall}\arrow[d] \arrow[r] & 0\\
	\begin{tabular}{c}$0$ \\ \\$0$\end{tabular}\arrow[r, shift left=5] \arrow[r, shift right=5] &\begin{tabular}{c}$E^{+1}$\\ $\directsum$ \\ $E^{-1}$\end{tabular}\arrow [r, shift left =5] \arrow [r, shift right =5]\arrow[d] & \begin{tabular}{ccc} $E^{++}\directsum E^{+-}$ \\$ \directsum  \quad \quad \directsum $\\ $E^{-+} \directsum E^{--}$ \end{tabular} \arrow[d, shift left=6] \arrow[d, shift right=6]\arrow [r, shift left =5] \arrow [r, shift right =5]& \begin{tabular}{c}$E^{+\forall}$\\$ \directsum$ \\ $E^{-\forall}$\end{tabular} \arrow[d] \arrow [r, shift left =5] \arrow [r, shift right =5] & \begin{tabular}{c}$0$ \\ \\$0$\end{tabular}\\
	0 \arrow[r] &E^{\forall1}\arrow[d] \arrow [r]& E^{\forall+}\arrow[d, shift left=6]  \directsum E^{\forall-} \arrow[d, shift right =6]\arrow[r] & E^{\forall\forall} \arrow[d]\arrow[r] & 0.\\
	& 0& 0\;\; \directsum   \;\;0 & 0 &
\end{tikzcd}
\]
\end{lemma}
\begin{proof}It follows from the definition of $E^\pm_\star$ that $$E^+_\star\cap E^-_\star= E_\star^1$$
(c.f. \cite[proof of Proposition~8.12]{kobayashi03}).
The result follows from a repeated application of \cite[Proposition~2.15]{kimspaper}.
\end{proof}

\begin{proposition}\label{prop:kimexact}
Let $n\ge0$. For $\circ,\moon\in\{+,-,1,\forall\}$, let $E^{\circ\moon}=E_\p^\circ\oplus  E_\q^\moon$ be as in Lemma~\ref{lem:bigdiag}. We write
\[
\Sel^{\circ\moon}=\ker\left(\Sel_p(E/K_n)\rightarrow\frac{\bigoplus_{\P|\p}H^1((K_n)_\P,E[p^\infty])}{E_\p^\circ\otimes\Qp/\Zp}\times \frac{\bigoplus_{\mathfrak{Q}|\q}H^1((K_n)_\mathfrak{Q},E[p^\infty])}{E^\moon_\q\otimes\Qp/\Zp} \right).
\]
Then, the all horizontal and vertical sequences of the following commutative diagram are exact:
\[
\begin{tikzcd}
	&0\arrow[d] & 0\arrow[d, shift left=6]  \;\; \directsum  \;\; 0\arrow[d, shift right=6] &  0 \arrow[d]&\\
	0 \arrow[r] &\Sel^{11} \arrow[d]\arrow [r]&\Sel^{1+}\arrow[d, shift left=6]  \directsum  \Sel^{1-}\arrow[d, shift right=6] \arrow[r] & \Sel^{1\forall}\arrow[d] \\
	\begin{tabular}{c}$0$ \\ \\$0$\end{tabular}\arrow[r, shift left=5] \arrow[r, shift right=5] &\begin{tabular}{c}$\Sel^{+1}$\\ $\directsum$ \\ $\Sel^{-1}$\end{tabular}\arrow [r, shift left =5] \arrow [r, shift right =5]\arrow[d] & \begin{tabular}{ccc} $\Sel^{++}\directsum \Sel^{+-}$ \\$ \directsum  \quad \quad \directsum $\\ $\Sel^{-+} \directsum \Sel^{--}$ \end{tabular} \arrow[d, shift left=6] \arrow[d, shift right=6]\arrow [r, shift left =5] \arrow [r, shift right =5]& \begin{tabular}{c}$\Sel^{+\forall}$\\$ \directsum$ \\ $\Sel^{-\forall}$\end{tabular} \arrow[d]  \\
	0 \arrow[r] &\Sel^{\forall1}\arrow [r]& \Sel^{\forall+} \directsum \Sel^{\forall-}\arrow[r] & \Sel^{\forall\forall} .
\end{tikzcd}\]
Furthermore, the cokernel of the  last map on all horizontal and vertical sequences are finite.
\end{proposition}
Note in particular that $\Sel^{\forall\forall}=\Sel_p(E/K_n)$.
\begin{proof}
This follows from the same proof of \cite[Proposition~10.1]{kobayashi03}.
\end{proof}

\begin{theorem}\label{thm:ap0}
 Let $E/\QQ$ be an elliptic curve with good supersingular reduction at $p$ and $a_p=0$. Then,
\[
\rank_{\Zp}\Sel_p(E/K_{n})^\vee=O(p^n)
\]
for $n\gg0$.
\end{theorem}
\begin{proof}
 Proposition~\ref{prop:kimexact} implies that
 \[
 \rank_{\Zp}\Sel_p(E/K_n)^\vee\le \sum_{\circ,\moon\in\{+,-\}}\rank_{\Zp}\Sel_p^{\circ\moon}(E/K_n)^\vee.
 \]
Theorem~\ref{kimscontroltheorem} allows us to  rewrite the right-hand side as
\[
 \sum_{\circ,\moon\in\{+,-\}}\rank_{\Zp}\Sel_p^{\circ\moon}(E/K_\infty)_{\langle\omega_n^\circ(\gamma_\p),\omega_n^\moon(\gamma_\q)\rangle}^\vee,
\]
 which is bounded by 
\[
\sum_{\circ,\moon\in\{+,-\}}\rank_{\Zp}\Sel_p^{\circ\moon}(E/K_\infty)_{\langle\omega_n(\gamma_\p),\omega_n(\gamma_\q)\rangle}^\vee \sum_{\circ,\moon\in\{+,-\}}\rank_{\Zp}\Sel_p^{\circ\moon}(E/K_\infty)_{\Gamma_n}^\vee
\]
since $\omega_n^\pm$ divide $\omega_n$.

Under the same notation of the proof of Theorem~\ref{thm:ordranks}, by \cite[proof of Proposition~8.4]{lp}, there is an isomorphism of $\Zp[[\Gamma_\cyc]]$-modules
\[
\Sel_p^{\circ\moon}(E/K_\infty)^{H_\cyc}\cong \Sel_p^\circ(E/\QQ_\cyc)\oplus \Sel_p^\moon(E^{(K)}/\QQ_\cyc).
\]
The two Selmer groups on the right-hand side are both $\Zp[[\Gamma_\cyc]]$-cotorsion (\cite[Theorem~2.2]{kobayashi03}). Hence, $\Sel_p^{\circ\moon}(E/K_\infty)^\vee$ is $\Lambda$-torsion and we are done by Theorem~\ref{thm:harris}.
\end{proof}

\section{The general supersingular case}

In this section, $E$ is supposed to be an elliptic curve with good supersingular at $p$ and $a_p\ne0$. In particular, $a_p=\pm p$.
\subsection{$\sharp/\flat$-Coleman maps and local properties}\label{S:local}
Let $F_\infty$ and $k_\cyc$ be the unramified $\Zp$-extension and the $\Zp$-cyclotomic extension of $\Qp$ respectively. Let $k_\infty$ be the compositum of $F_\infty$ and $k_\cyc$. For $n\ge 0$, $k_n$ and $F_n$ denotes the sub-extensions of $k_\infty$ and $F_\infty$ such that $[k_n:\Qp]=p^n$ and $[F_n:\Qp]=p^n$ respectively. Let  $\Gamma_p=\Gal(k_\infty/\Qp)$, $\Gamma_\ur=\Gal(F_\infty/\Qp)=\langle\sigma\rangle$ and $\Gamma_\cyc=\Gal(k_\cyc/\Qp)=\langle \gamma\rangle$, where $\gamma$ is a  fixed topological generator of $\Gamma_\cyc$ and $\sigma$ is the arithmetic Frobenius.

We recall from \cite{sprung12,llz0} that for $\bullet\in\{\sharp,\flat\}$, there is a Coleman map $\Col_\bullet^\cyc:H^1(\Qp,T\otimes\Zp[[\Gamma_\cyc]])\rightarrow\Zp[[\Gamma_\cyc]]$ decomposing Perrin-Riou's big logarithm map
\[
\LL_{\cyc}=\begin{pmatrix}
\omega&\vp(\omega)
\end{pmatrix}\Mlog\begin{pmatrix}
 \Col_\sharp^\cyc\\\Col_\flat^\cyc
 \end{pmatrix},
\]
 where $\omega$ is a basis of $\Fil^0\Dcris(T)$, $\Mlog$ is the logarithmic matrix defined as
\[
\lim_{n\rightarrow\infty}A^{n+1}C_n\cdots C_1,
\]
with $A=\begin{pmatrix}
0&\frac{-1}{p}\\
1&\frac{a_p}{p}
\end{pmatrix}$ and $C_n=\begin{pmatrix}
a_p&1\\
\Phi_n(\gamma)&0
\end{pmatrix}$.

As shown in \cite[\S2.3]{BL2}, we may take inverse limits of the Coleman maps over $F_n$ to define two-variable versions of these maps, decomposing the two-variable big logarithm map of Loeffler-Zerbes in \cite{LZ0}. Furthermore, on choosing a basis of the Yager module defined in \cite[\S3]{LZ0}, we may assume that the Coleman maps land inside $\Zp[[\Gamma_p]]$. More precisely, we have
\[
\Col_\sharp,\Col_\flat:H^1(\Qp,T\otimes\Zp[[\Gamma_p]])\rightarrow\Zp[[\Gamma_p]],
\]
such that
\begin{equation}\label{eq:decomp}
 \LL=\begin{pmatrix}
 \omega&\vp(\omega)
 \end{pmatrix}\Mlog\begin{pmatrix}
 \Col_\sharp\\\Col_\flat
 \end{pmatrix},
\end{equation}
 where $\LL:H^1(\Qp,T\otimes \Zp[[\Gamma_p]])\rightarrow \calH(\Gamma_p)\otimes \Dcris(T)$ is Loeffler-Zerbes' big logarithm map multiplied by the chosen basis of the Yager module. See  \cite[\S2]{sprung16} and \cite[\S2.3]{BL2}.

 Let $F= F_m\cdot k_n$ for some integers $m$ and $n$. The Coleman maps can be used to study Bloch-Kato's Selmer condition over $F$, which is commonly denoted by $H^1_f(F,T)$. Note that in the case of elliptic curves, there is an isomorphism $H^1_f(F,T)\cong \hat{E}(\mathfrak{M}_F)$, where $\hat{E}$ denotes the formal group attached to $E$ at $p$ and $\mathfrak{M}_F$ denotes the maximal ideal of the ring of integers of $F$ (c.f. \cite[Example 3.11]{BK}).
 \begin{proposition}\label{prop:H1f}
 Let $z\in \HIw(\Qp,T\otimes \Zp[[\Gamma_p]])$ and $\eta$ a  character on $\Gamma_p$ whose cyclotomic part is of conductor $p^{n+1}>1$. If we write $e_\eta$ for the idempotent corresponding to the character $\eta$, then, $e_\eta\cdot z$ lies inside $e_\eta\cdot H^1_f(F,T)$ if and only if the element in 
\begin{equation}\label{eq:evaluate}
 H_{\sharp,n}\Col_\sharp(z)^{\sigma^{-n-1}}+H_{\flat,n} \Col_\flat(z)^{\sigma^{-n-1}}\in\Zp[[\Gamma_p]]
\end{equation} vanishes at $\eta$, where $H_{\sharp,n}$ and $H_{\flat,n}$ are given by the entries of the first row of
 \[
 C_n\cdots C_1.
 \]
 \end{proposition}
\begin{proof}
Let $\exp^*:H^1(F,T)\rightarrow F\otimes\Dcris(T)$ be Bloch-Kato's dual exponential map. Then, its kernel is exactly $H^1_f(F,T)$. Recall from \cite[Theorem~4.15]{LZ0} that
\[
\Phi^{-n-1}\LL(z)(\eta)=\varepsilon(\eta^{-1})\exp^*(e_\eta\cdot z),
\]
where $\varepsilon(\eta^{-1})$ denotes the $\varepsilon$-factor of $\eta^{-1}$ and $\Phi$ is the operator on $F_m\otimes\Dcris(T)$ which act as  $\sigma$ on $F_m$ and $\vp$ on $\Dcris(T)$. Hence, $e_\eta\cdot z$ lies inside $e_\eta\cdot H^1_f(F,T)$ if and only if $\Phi^{-n-1}\LL(z)(\eta)=0$.

By definition $\Mlog$ is invariant under $\sigma$. Furthermore, \cite[Lemma~3.7]{llz1} tells us that 
\[
\Mlog(\eta)= A^{n+1}C_n\cdots C_1(\eta).
\]
If we combine this with \eqref{eq:decomp}, we have
\[
\Phi^{-n-1}\LL(z)(\eta)=\begin{pmatrix}
\omega&\vp(\omega)
\end{pmatrix} C_n\cdots C_1(\eta)\begin{pmatrix}
\Col_\sharp(z)^{\sigma^{-n-1}}(\eta)\\ \Col_\flat(z)^{\sigma^{-n-1}}(\eta)
\end{pmatrix}.
\]
Recall from the definition of $C_n$ that the second row of $C_n(\eta)$ vanishes as $\eta$ sends $\gamma$ to a primitive $p^n$-th root of unity. Consequently, the expression above simplifies to \eqref{eq:evaluate} and we are done.
\end{proof}

\subsection{$\sharp/\flat$-Selmer groups}
Recall that $\gamma_1,\ldots,\gamma_{p^t}$ is a chosen set of coset representatives of $\Gamma/\Gamma_\p$. This gives the following decomposition
\[
H^1(K_\p,T\otimes\Lambda)=\bigoplus_{i=1}^{p^t}H^1(K_\p,T\otimes\Zp[[\Gamma_\p]])\cdot \gamma_i.
\]
Note that we may identify $K_\p$ with $\Qp$ and $\Gamma_\p$ with $\Gamma_p$. We define the $\sharp/\flat$-Coleman maps at $\p$ by
\begin{align*}
\Col_{\p,\bullet}:H^1(K_\p,T\otimes\Lambda)&\rightarrow \Lambda\\
x=\sum x_i\cdot \gamma_i&\mapsto \sum \Col_\bullet(x_i)\cdot \gamma_i.
\end{align*}

For $\bullet\in\{\sharp,\flat\}$, we define 
\[
E^\bullet_\p\subset \bigoplus_{i=1}^{p^t} H^1(K_{\infty,\gamma_i},E[p^\infty])
\] to be the orthogonal complement of $\ker \Col_{\p,\bullet}$ under local Tate duality. Similarly, we may define $\Col_{\q,\bullet}$ and $E^\bullet_\q$ in the same way.
\begin{remark}\label{rk:local}
 The embeddings we fixed at the beginning mean that when we localize at $\p$, the Galois group $G_\p$ is sent to $\Gamma_\cyc$, whereas $G_\q$ is sent to $\Gamma_\ur$. We may choose our topological generators so that $\gamma_\p$ and $\gamma_\q$ are sent to $\gamma$ and $\sigma$ respectively. When we localize at $\q$, the images as reversed.
\end{remark}

Let $N$ be the conductor of $E$ and let $S$ be the set of primes dividing $pN$. Given a number field $F$, we write $H^1_S(F,-)$ for the Galois cohomology of the Galois group of the maximal algebraic extension of $F$ that is unramified outside $S$.
Recall that the classical $p$-Selmer group over $K_\infty$ is defined to be
\[
\Sel_p(E/K_\infty)=\ker\left(H^1_S(K_\infty,E[p^\infty])\rightarrow\bigoplus_{v|pN}\frac{H^1(K_{\infty,v},E[p^\infty])}{E(K_{\infty,v})\otimes\Qp/\Zp} \right).
\]
We  define four $\sharp/\flat$-Selmer groups 
\begin{equation}
\Sel_p^{\circ\moon}(E/K_\infty)=\ker\left(\Sel_p(E/K_\infty)\rightarrow \frac{ \bigoplus_{i=1}^{p^t} H^1(K_{\infty,\gamma_i},E[p^\infty])}{E^\circ_\p}\oplus\frac{ \bigoplus_{i=1}^{p^t} H^1(K_{\infty,\delta_i},E[p^\infty])}{E^\moon_\q}\right)
\end{equation}
for $\circ,\moon\in\{\sharp,\flat\}$.
We shall from now on assume the following hypothesis holds.

\vspace{0.5cm}
\begin{itemize}
\item[\textbf{(H.tor)}] At least one of the four Selmer groups $\Sel_p^{\circ\moon}(E/K_\infty)$, $\circ,\moon\in\{\sharp,\flat\}$, is $\Lambda$-cotorsion.
\end{itemize}
\vspace{0.5cm}

These Selmer groups can be linked to certain $2$-variable $p$-adic $L$-functions constructed in \cite{lei14}. When the latter is non-zero, then (H.tor) holds (see \cite[\S3.4]{CCSS}). 

These Selmer groups generalize the signed Selmer groups of Kim. However, we do not know whether an analogue for Proposition~\ref{prop:kimexact} holds for these Selmer groups. Consequently, we cannot generalize the proof of Theorem~\ref{thm:ap0} directly to this setting. Instead, we make use of ideas from \cite{sprung13,llz1} to estimate the ranks of $E$ over $K_\infty$.

Let
\[
\Sel_p^0(E/K_\infty)=\ker\left(\Sel_p(E/K_\infty)\rightarrow  \bigoplus_{v|p} H^1(K_{\infty,v},E[p^\infty]) \right)
\]
be the strict Selmer group over $K_\infty$. 
We recall the Poitiou-Tate exact sequence
\begin{align}
\label{eq:PT}\HIw(K_\infty,T)\rightarrow \frac{H^1(K_\p,T\otimes\Lambda)}{\ker\Col_{\p,\circ}}\oplus \frac{H^1(K_\q,T\otimes\Lambda)}{\ker\Col_{\q,\moon}}\rightarrow&\\
 \Sel^{\circ\moon}(E/K_\infty)^\vee\rightarrow& \Sel^{0}(E/K_\infty)^\vee\rightarrow 0,\notag
\end{align}
where $H^i_{\mathrm{Iw}}(K_\infty,T)$ is defined a the inverse limit $\displaystyle \varprojlim H^i_S(K_n,T)$ (see for example \cite[(7.20)]{kobayashi03}).

\begin{lemma}\label{lem:finetor}
Under the hypothesis \textbf{(H.tor)}, the Iwasawa module $ \Sel_p^{0}(E/K_\infty)^\vee$ is $\Lambda$-torsion, whereas $H^1_{\mathrm{Iw}}(K_\infty,T)$ is of rank $2$.
\end{lemma}
\begin{proof}
The torsionness of $ \Sel_p^{0}(E/K_\infty)^\vee$ follows from that of $\Sel_p^{\circ\moon}(E/K_\infty)^\vee$ and \eqref{eq:PT}. The statement about $\HIw(K_\infty,T)$ then follows from the weak Leopoldt conjecture, c.f. \cite[Corollary~A.8]{LZ0}. 
\end{proof}

\begin{remark}\label{rk:torsion}
Let $c_1,c_2\in\HIw(K_\infty,T)$ be any elements such that $\HIw(K_\infty,T)/\langle
c_1,c_2\rangle$ is $\Lambda$-torsion. Then, the cotorsionness of $\Sel^{\circ\moon}(E/K_\infty)$ together with \eqref{eq:PT} imply that
\[
\det\begin{pmatrix}
\Col_{\p,\circ}(\loc_\p(c_1))&\Col_{\q,\moon}(\loc_\q(c_1))\\
\Col_{\p,\circ}(\loc_\p(c_2))&\Col_{\q,\moon}(\loc_\q(c_2))
\end{pmatrix}\ne 0.
\]
We shall denote this determinant by $\det\Col_{\circ\moon}(c_1\wedge c_2)$ from now on.
\end{remark}

\subsection{Classical Selmer groups and Coleman maps}
In this subsection, we relate modified Selmer-type groups to Coleman maps. The following cohomology groups were introduced when studying the growth of Sha in \cite{kobayashi03} and \cite{sprung13}:

For $n\ge 0$, we define
\begin{align*}
\Y_n'&=\coker\left(\HIw(K_\infty,T)_{\Gamma_n}\rightarrow \prod_{v|p}H^1_{/f}(K_{n,v},T)\right),\\
\Y_n&=\coker\left(H^1_S(K_n,T)\rightarrow \prod_{v|p}H^1_{/f}(K_{n,v},T)\right).
\end{align*}
We also define the $\X^0_n$ and $\X_n$ to be the Pontryagin dual of the strict Selmer group $\Sel_p^0(E/K_n)$
and the classical Selmer group $\Sel_p(E/K_n)$ respectively. Then, we have the Poitou-Tate exact sequence over $K_n$ gives the following exact sequence:
\[
0\rightarrow \Y_n\rightarrow \X_n\rightarrow \X_n^0\rightarrow 0
\]
(see for example \cite[(10.35)]{kobayashi03}).

\begin{proposition}
The $\Zp$-rank of $\X_n^0$ is $O(p^n)$.
\end{proposition}
\begin{proof}
We first show that  the strict Selmer group satisfy a control theorem.
Consider the following commutative diagram
\begin{equation} \label{diag:control}
			\begin{tikzcd}
				0 \arrow{r} & \Sel_p^0(E/K_n) 
				\arrow{r}
				\arrow[d,"\alpha_n"] & H^1(K_n,E[p^\infty])
				\arrow{r} \arrow[d,"\beta_n"] & \prod_{v_n}
				H^1(K_{n,v_n},E[p^\infty])
				\arrow[d,"\prod\gamma_{v_n}"]\\
				0 \arrow{r} & 
				\Sel_p^0(E/K_\infty)^{\Gamma_n} 
				\arrow{r}
				& 
				H^1(K_\infty,E[p^\infty])^{\Gamma_n}
				\arrow{r} & \prod_{w}
		H^1(K_{\infty,w},E[p^\infty])^{\Gamma_n},
		\end{tikzcd}
	\end{equation}
	where the vertical maps are restriction maps, the product on the first row runs over all places $v_n$ of $K_n$, whereas that on the second row runs over all places $w$ of $K_\infty$. By definitions, both rows are exact. Note that $E(K_{\infty,w})[p^\infty]=0$ for all places $w$ of $K_\infty$ such that $w|p$ (c.f. \cite[Proposition~3.1]{KO}). Therefore, $\gamma_{v_n}$ is an isomorphism for all $v_n|p$ by the inflation-restriction exact sequence. Similarly, $\beta_n$ is also an isomorphism. If $v_n\nmid p$, then the proof of \cite[Theorem 2]{greenberg} tells us that $\ker\gamma_{v_n}$ is finite of bounded order (see also \cite[Proposition~2.19]{kimspaper}). Therefore, on applying the snake lemma to \eqref{diag:control}, we deduce that $\ker\alpha_n$ and $\coker\alpha_n$ are finite and of orders bounded independently of $n$.
	
	Lemma~\ref{lem:finetor} tells us that $\X_n^0$ is $\Lambda$-torsion. Hence, Theorem~\ref{thm:harris}, together with the control theorem above imply that
	\[
	\rank_{\Zp}\X_n^0=\rank_{\Zp}\left(\X_\infty^0\right)_{\Gamma_n}=O(p^n),
	\]
	as required.
\end{proof}

Therefore, in order to study $\rank_{\Zp}\X_n$, it is enough to study $\rank_{\Zp}\Y_n$. Furthermore, there is a natural surjection $\Y'_n\rightarrow \Y_n$ (c.f. \cite[diagram (10.36)]{kobayashi03}). In particular, 
\[
\rank_{\Zp}\Y_n\le \rank_{\Zp}\Y_n'
\]
for all $n\ge0$ and we are reduced to estimate $\rank_{\Zp}\Y_n'$.
 We shall do so via the $\sharp/\flat$-Coleman maps.  Let us first recall the following calculations from \cite{CM}. We shall write $W=\mu_{p^\infty}\times \mu_{p^\infty}$. Given $w\in W$, $o(w)$ denotes the least integer $r$ such that $w^{p^r}=(1,1)$. Given $w=(w_1,w_2)\in W$, we shall write $\Zp[w]=\Zp[w_1,w_2]$. If $F\in\Lambda$, we may evaluate $F$ at $w$ in the following way. We may rewrite $F$ as a power series in $\gamma_\p-1$ and $\gamma_\q-1$, say $F_0(\gamma_\p-1,\gamma_\q-1)$. Then, we define
\[
F(w)=F_0(w_1-1,w_2-1)\in \Zp[w].
\]
Given any $\Lambda$-module $M$, we write $M_w=M\otimes \Zp[w]$ for the $\Zp[w]$-module induced by this evaluation map.

\begin{lemma}\label{lem:quotientranks}
Let $n\ge0$ be an integer and $M$ a $\Lambda$-module. Then, 
\[
\rank_{\Zp}M_{\Gamma_n}=\rank_{\Zp}\bigoplus_{w}M_w,
\]
where the direct sum runs over the conjugacy classes of $w\in W$ such that $w^{p^n}=(1,1)$.
\end{lemma}
\begin{proof}
This is an immediate consequence of \cite[Lemma~ 2.7]{CM}, which says that the kernel and cokernel of the morphism $M_{\Gamma_n}\rightarrow\bigoplus_{w}M_w$ are both finite. 
\end{proof}

From now on, we fix two elements $c_1,c_2$ of $\HIw(K_\infty,T)$ such that  $\HIw(K_\infty,T)/\langle
c_1,c_2\rangle$ is $\Lambda$-torsion (c.f Remark~\ref{rk:torsion}). Let us write $\loc_{p,n}(c_i)$ for the image of $c_i$ inside  $H^1_{/f}(K_{n,p},T):=\prod_{v|p}H^1_{/f}(K_{n,v},T)$. 
\begin{corollary}
Let $n\ge0$ be an integer. Then,
\begin{equation}\label{eq:growthrank}
\rank_{\Zp}\Y_{n+1}'-\rank_{\Zp}\Y_n'\le\sum_\theta \rank_{\Zp}\left(e_\theta\cdot \frac{H^1_{/f}(K_{n,p},T)}{\langle\loc_{p,n}(c_1),\loc_{p,n}(c_2)\rangle}\right),
\end{equation}
where the direct sum runs over conjugacy classes of  characters on $\Gamma_{n+1}$ that do not factor through $\Gamma_n$. 
\end{corollary}
\begin{proof}
Note that $\Y_n'$ is a quotient of $\frac{H^1_{/f}(K_{n,p},T)}{\langle\loc_{p,n}(c_1),\loc_{p,n}(c_2)\rangle}$. Therefore, the inequality follows from Lemma~\ref{lem:quotientranks} and the fact that the functor $M\rightarrow M_w$ for $w=(w_1,w_2)\in W$ is given by $M\rightarrow e_\eta\cdot M$, where $\theta$ is the character on $\Gamma$ that sends $\gamma_\p$ and $\gamma_\q$ to $w_1$ and $w_2$ respectively.
\end{proof}

 Let $\theta$ be a character of $\Gamma$ of conductor $\p^{r+1}\q^{s+1}$. When restricted to $\Gamma_\p$, Remark~\ref{rk:local} tells us that  the character $\theta$ gives a character on $\Gamma_\p$ (resp. $\Gamma_\q$) whose cyclotomic part is of conductor $p^{r+1}$ (resp. $p^{s+1}$). The element in \eqref{eq:evaluate} then leads us to introduce the following notation. For $\star\in\{\p,\q\}$, $z\in H^1(K_\star,T\otimes\Lambda)$ and $t\ge0$, we write
 \[
 \uCol_{\star,t}(z)=H_{\star,\sharp,t}\Col_{\star,\sharp}(z)^{\sigma^{-t-1}}+H_{\star,\flat,t}\Col_{\star,\flat}(z)^{\sigma^{-t-1}},
 \]
 where $H_{\star,\sharp,t}$ and $H_{\star,\flat,t}$ are the entries of the first row in the matrix obtained from replacing $\gamma$ by $\gamma_\star$ in the product $C_{t}\cdots C_1$.

\begin{lemma}\label{lem:countlem}
Let $\theta$ be a character of $\Gamma$ of conductor $\p^{r+1}\q^{s+1}$. Then $$\rank_{\Zp}\left(e_\theta\cdot \frac{H^1_{/f}(K_{n,p},T)}{\langle\loc_{p,n}(c_1),\loc_{p,n}(c_2)\rangle}\right)>0$$ if and only if 
\[
\det\begin{pmatrix}
\uCol_{\p,r}(\loc_\p(c_1))&\uCol_{\q,s}(\loc_\q(c_1))\\
\uCol_{\p,r}(\loc_\p(c_2))&\uCol_{\q,s}(\loc_\q(c_2))
\end{pmatrix}
\]
vanishes at $\theta$. When the aforementioned determinant does vanish at $\theta$, the $\Zp$-rank of $e_\theta\cdot \frac{H^1_{/f}(K_{n,p},T)}{\langle\loc_{p,n}(c_1),\loc_{p,n}(c_2)\rangle}$ is either $\dim_{\Qp}\Qp(\theta)$ or $2\dim_{\Qp}\Qp(\theta)$, where $\Qp(\theta)$ is the extension of $\Qp$ generated by the image of $\theta$.
\end{lemma}
\begin{proof}
For $t\in\{r,s\}$ and $\star\in\{\p,\q\}$, the morphism $\uCol_{\star,t}$ induces 
\[
e_\theta\cdot \prod_{v|\star}H^1_{/f}(K_{n,v},T)\otimes\Qp\cong \Qp(\theta)
\]
by Proposition~\ref{prop:H1f}. Consequently, if we write
\[
v_1=\left(\uCol_{\p,r}(\loc_\p(c_1)),\uCol_{\q,s}(\loc_\q(c_1))\right),\quad v_2=\left((\uCol_{\p,r}(\loc_\p(c_2)),\uCol_{\q,s}(\loc_\q(c_2))\right),
\]
then 
\[
e_\theta\cdot \frac{H^1_{/f}(K_{n,p},T)}{\langle\loc_{p,n}(c_1),\loc_{p,n}(c_2)\rangle}\otimes\Qp\cong \frac{\Qp(\theta)^{\oplus 2}}{\langle v_1,v_2\rangle}
\]
as $\Qp(\theta)$-vector spaces. Hence the result.
\end{proof}

For simplicity, we shall write the determinant above as $\det\uCol_{r,s}(c_1\wedge c_2)$. A direct calculation tells us that this is related to the determinants $\det\Col_{\circ\moon}(c_1\wedge c_2)$, $\circ,\moon\in\{\sharp,\flat\}$, via the following equation:
\begin{equation}\label{eq:formula}
\det\uCol_{r,s}(c_1\wedge c_2)(\theta)=\sum_{\circ,\moon\in\{\sharp,\flat\}}H_{\p,\circ,r}H_{\q,\moon,s}\det\uCol_{\circ\moon}(c_1\wedge c_2)(\theta).
\end{equation}
We shall study the $p$-adic valuations of the summands on the right-hand side of \eqref{eq:formula} to show that $\det\uCol_{r,s}(c_1\wedge c_2)(\theta)\ne 0$ for a large family of characters $\theta$.

\subsection{$p$-adic valuations of power series}

Given an element $\zeta\in \mu_{p^\infty}$, we shall write $\ord(\zeta)$ for the smallest integer $t$ such that $\zeta^{p^{t}}=1$.
\begin{proposition}
Let $F(X,Y)\in \Zp[[X,Y]]$ be a non-zero power series. There exist integers $a,b_1,b_2,c_1,c_2$ and $n_0$ such that for all $w=(w_1,w_2)\in W\setminus \left(\mu_{p^{n_0}} \times \mu_{p^{n_0}}\right) $, 
\[
\ord_p (F(w))=a+\frac{b_i}{\phi(p^m)}+\frac{c_i}{\phi(p^n)}
\]
with $i=1$ or $2$, whenever $|\ord(w_1)-\ord(w_2)|\gg0$.
\end{proposition}
\begin{proof}
Let $w=(w_1,w_2)$ and suppose that $\ord(w_1)=m$ and $\ord(w_2)=n$. We assume that $F(X,Y)=p^a G(X,Y)$ for some integer $a\ge0$ and $G(X,Y)\in \Zp[[X,Y]]\setminus p\Zp[[X,Y]]$ is irreducible.

Following \cite[proof of Proposition~2.9]{cuoco}, there exist integers $a,b\ge0$ independent of $m$ such that whenever $m\gg0$, we have the Weierstrass preparation
\[
G(w_1-1,Y)=u\times (w_1-1)^b(Y^c+d_{c-1}Y^{b-1}+\cdots d_1Y+d_0),
\]
where $u\in \Zp[w_1][Y]^\times$ and $d_i\in (w_1-1)\Zp[w_1]$. In particular, $\ord_p(d_i)\ge \frac{1}{\phi(p^m)}$ for all $i$. Consequently, we have
\[
\ord_p(G(w))=\frac{b}{\phi(p^m)}+\frac{c}{\phi(p^n)},
\]
whenever  $\frac{c}{\phi(p^n)}<\frac{1}{\phi(p^m)}$ (which can be ensured by $n-m\gg0$).

Similarly, when $m-n\gg0$, we may evaluate $\ord_p(G(w))$ by considering the Weierstrass preparation for $G(X,w_2-1)$.
\end{proof}

This can be translated to a result on elements in the Iwasawa algebra $\Lambda$ as follows.
\begin{corollary}\label{cor:valueC}
Let $F\in \Lambda$ be a non-zero element. There exist integers $a,b_1,b_2,c_1,c_2$ and $n_0$ such that whenever  $\theta$ is a character of $\Gamma$ of conductor $\p^{r+1}\q^{s+1}$ with $r,s\ge n_0$ and $|r-s|\gg0$, we have
\[
\ord_p (F(\theta))=a+\frac{b_i}{\phi(p^r)}+\frac{c_i}{\phi(p^s)}
\]
where $i=1$ or $2$ depending on whether $r>s$ or $s >r$.
\end{corollary}

Let us now recall the following result from \cite[Lemma~4.5]{sprung13} and \cite[Proposition~4.6]{llz1}.

\begin{proposition}\label{prop:valueH}
Let $\eta$ be a character of $\Gamma_\cyc$ of conductor $p^{n+1}>1$ and let $H_{\sharp,n}$ and $H_{\flat,n}$ be as defined in \eqref{eq:evaluate}. Then,
\[
\ord_p(H_{\sharp,n}(\eta))=\begin{cases}
1+\sum_{i=1}^{\frac{n-1}{2}}\frac{1}{p^{2i-1}}&n\in2\ZZ+1,\\
\sum_{i=1}^{\frac{n}{2}}\frac{1}{p^{2i-1}}&n\in2\ZZ,
\end{cases}
\]
whereas
\[
\ord_p(H_{\flat,n}(\eta))=\begin{cases}
\sum_{i=1}^{\frac{n-1}{2}}\frac{1}{p^{2i}}&n\in2\ZZ+1,\\
1+\sum_{i=1}^{\frac{n}{2}-1}\frac{1}{p^{2i}}&n\in2\ZZ.
\end{cases}
\]
\end{proposition}

We now combine all these calculations to study the $p$-adic valuations of the summands that appeared in \eqref{eq:formula}.
\begin{proposition}\label{prop:nogrowth}
Let $\theta$ be a character of $\Gamma$ of conductor $\p^{r+1}\q^{s+1}$. Let $\circ,\moon,\triangledown,\blacktriangledown\in\{\sharp,\flat\}$ with $(\circ,\moon)\ne(\triangledown,\blacktriangledown)$ and that at least one of $\det\uCol_{\circ\moon}(c_1\wedge c_2)(\theta)$ and $\det\uCol_{\triangledown,\blacktriangledown}(c_1\wedge c_2)(\theta)$ is non-zero. Then,
\[
\ord_p(H_{\p,\circ,r}H_{\q,\moon,s}\det\uCol_{\circ\moon}(c_1\wedge c_2)(\theta))\ne
\ord_p(H_{\p,\triangledown,r}H_{\q,\blacktriangledown,s}\uCol_{\triangledown,\blacktriangledown}(c_1\wedge c_2)(\theta))
\]whenever $r,s, |r-s|\gg0$.
\end{proposition}
\begin{proof}
If either $\det\uCol_{\circ\moon}(c_1\wedge c_2)(\theta)$ or $\det\uCol_{\triangledown,\blacktriangledown}(c_1\wedge c_2)(\theta)$  is zero, then the result is immediate. We may therefore assume that both values are non-zero.

Let us first consider the case where $\circ=\triangledown$ but $\moon\ne\blacktriangledown$. Let's say $\moon=\sharp$ and $\blacktriangledown=\flat$, $r>s$ and $s$ is odd. We would like to show that 
\[
\ord_p(H_{\q,\sharp,s}\det\uCol_{\circ,\sharp}(c_1\wedge c_2)(\theta))\ne
\ord_p(H_{\q,\flat,s}\det\uCol_{\circ,\flat}(c_1\wedge c_2)(\theta)).
\]
Corollary~\ref{cor:valueC} and Proposition~\ref{prop:valueH} tell us that
\begin{align*}
\ord_p(H_{\q,\sharp,s}\det\uCol_{\circ,\sharp}(c_1\wedge c_2)(\theta))&=1+\sum_{i=1}^{\frac{s-1}{2}}\frac{1}{p^{2i-1}}+a+\frac{b}{\phi(p^r)}+\frac{c}{\phi(p^s)},\\
\ord_p(H_{\q,\flat,s}\det\uCol_{\circ,\flat}(c_1\wedge c_2)(\theta))&=\sum_{i=1}^{\frac{s-1}{2}}\frac{1}{p^{2i}}+d+\frac{e}{\phi(p^r)}+\frac{f}{\phi(p^s)},
\end{align*}
for some integers  $a,b,c,d,e,f$ that are independent of $r$ and $s$ as long as $r,s,r-s\gg0$. The difference of these two quantities is given by
\[
\ord_p(a_p)+a-d+\frac{b-e}{\phi(p^r)}+\frac{c-f}{\phi(p^s)}+\frac{1-p^{1-s}}{p+1}.
\]
The appearance of $p+1$ in the denominator of the last term implies that this is non-zero. The cases $s$ is even or $s>r$ can be proved similarly. The same proof also works when $\moon=\blacktriangledown$ and  $\circ\ne\triangledown$.

It remains to consider the case where $\moon\ne\blacktriangledown$ and  $\circ\ne\triangledown$. Suppose that  $(\circ,\moon)=(\sharp,\flat)$, $(\triangledown,\blacktriangledown)=(\flat,\sharp)$ and that both $r$ and $s$ are even. Then, the difference
\[
\ord_p(H_{\p,\sharp,r}H_{\q,\flat,s}\det\uCol_{\sharp,\flat}(c_1\wedge c_2)(\theta))-
\ord_p(H_{\p,\flat,r}H_{\q,\sharp,s}\det\uCol_{\flat,\sharp}(c_1\wedge c_2)(\theta))
\]
is of the form
\[
*+\frac{*}{\phi(p^r)}+\frac{*}{\phi(p^s)}+\frac{p^{1-s}-p^{1-r}}{p+1},
\]
where the $*$'s represent some integers. Once again, this cannot be zero because of the presence of $p+1$ in the denominator of the last term as long as $r\ne s$. Similar calculations show that the same is true for the other combinations of parities of $r$ and $s$. The same proof also works when  $(\circ,\moon)=(\sharp,\sharp)$, $(\triangledown,\blacktriangledown)=(\flat,\flat)$.
\end{proof}

\begin{corollary}\label{cor:countchar}Suppose \textbf{(H.tor)} holds.
Let $n\ge1$ be an integer and write $\Xi_n$ for the set of characters $\theta$ on $\Gamma$ which factor through $\Gamma_{n}$ but not $\Gamma_{n-1}$ such that 
$$
\det\uCol_{r,s}(c_1\wedge c_2)(\theta)=0,
$$
(with the conductor of $\theta$ being $\p^{r+1}\q^{s+1}$). Then, the cardinality of $\Xi_n$ is bounded independently of $n$.
\end{corollary}
\begin{proof}
Recall from \eqref{eq:formula} that
\[
\det\uCol_{r,s}(c_1\wedge c_2)(\theta)=\sum_{\circ,\moon\in\{\sharp,\flat\}}H_{\p,\circ,r}H_{\q,\moon,s}\det\uCol_{\circ\moon}(c_1\wedge c_2)(\theta).
\]
if $\theta$ is a character on $\Gamma$ of conductor $\p^{r+1}\q^{s+1}$. If $\theta$ factors through $\Gamma_{n}$ but not $\theta_{n-1}$, then either  $r$ or $s$ equals $n$. 

If $\Sel^{\circ\moon}(E/K_\infty)^\vee$ is $\Lambda$-torsion, then Remark~\ref{rk:torsion} says that $\det\uCol_{\circ\moon}(c_1\wedge c_2)\ne0$.  Therefore, \textbf{(H.tor)} ensures that when $r,s\gg0$, at least one of the four determinants $\det\uCol_{\circ\moon}(c_1\wedge c_2)(\theta)$ does not vanish. We may therefore apply Proposition~\ref{prop:nogrowth}, which tells us that for $r,s,|r-s|\gg0$, the non-zero summands of the four-term sum have distinct $p$-adic valuations. 

In short, there exists a fixed integer $n_0$ (independent of $n$), such that if $\det\uCol_{r,s}(c_1\wedge c_2)(\theta)=0$, then  $r,s,|r-s|\le n_0$. If either $r$ or $s$ equals $n$, then the number of such $(r,s)$ is  clearly bounded. Hence we are done.
\end{proof}

We can now give the following bound on $\rank_{\Zp}\Y'_n$.

\begin{proposition}
$\rank_{\Zp}\Y_n'=O(p^n)$.
\end{proposition}
\begin{proof}
Lemma~\ref{lem:countlem} tells us that
\[
\rank_{\Zp}\Y'_{n}-\rank_{\Zp}\Y'_{n-1}\le 2C_n\phi(p^{n}),
\]
where $C_n$ is the number of characters $\theta$ that factor through $\Gamma_{n}$ but not $\Gamma_{n-1}$ and that $\det \uCol_{r,s}(c_1\wedge c_2)(\theta)=0$. Corollary~\ref{cor:countchar} says that there exists a constant $C$ such that $C_n\le C$ for all $n$. On summing the above inequality over $n$, we have
\[
\rank_{\Zp}\Y'_{n}-\rank_{\Zp}\Y'_{1}\le \sum_{i=2}^nC\phi(p^{i})=O(p^n)
\]
as required.
\end{proof}

%\bibliographystyle{amsalpha}
%\bibliography{references}

\end{document}